\documentclass[a4paper,10pt]{article}
\usepackage{pgf,tikz}
\usetikzlibrary{arrows}
\usetikzlibrary{decorations.pathreplacing}
\usepackage{latexsym}
\usepackage{amsfonts, amsmath, amssymb, amsthm}
\usepackage[utf8]{inputenc}
\usepackage[T1]{fontenc}
\usepackage[english]{babel}
\usepackage{stmaryrd}
\usepackage{subfig}
\usepackage{frcursive}
\usepackage{mathrsfs}
\usepackage{enumitem}
\newcommand{\N}{\mathbb{N}}

\newcommand{\R}{\mathbb{R}}

\newcommand{\lang}{\mathcal{L}}
\newcommand{\A}{\mathcal{A}}

\newtheorem{thm}{Theorem}
\newtheorem*{thm*}{Theorem}
\newtheorem{lemma}{Lemma}[section]
\newtheorem{prop}[lemma]{Proposition}
\newtheorem*{prop*}{Proposition}
\newtheorem{cor}[lemma]{Corollary}
\newtheorem{defi}[lemma]{Definition}
\newtheorem{theorem}[lemma]{Theorem}

\theoremstyle{definition}
\newtheorem{rem}[lemma]{Remark}
\newtheorem{ex}[lemma]{Example}

\title{Thermodynamic formalism and $k$-bonacci substitutions}
\author{Jordan Emme \thanks{Aix-Marseille Université, CNRS, Centrale Marseille, I2M, UMR 7373, 13453 Marseille, France. E-mail: jordan.emme@univ-amu.fr}}
\date{}
\begin{document}

\maketitle
\begin{abstract}
We study $k$-bonacci substitutions through the point of view of thermodynamic formalism. For each substitution we define a renormalization operator associated to it and examine its iterates over potentials in a certain class. We also study the pressure function associated to potentials in this class and prove the existence of a freezing phase transition which is realized by the only ergodic measure on the subshift associated to the substitution.
\end{abstract}

\section{Introduction}
\subsection{Background}

Given a dynamical system $(X,T)$ and a continuous function $V:X\rightarrow\R_+$ called the potential, we define the pressure function:
$$
\forall \beta \in \R_+,\, P(\beta)=\sup\left\{h_{\mu} - \beta\int_X V d\mu\right\},
$$
the sup being taken over all $T$-invariant probability measures on $X$ and $h_{\mu}$ being the Kolmogorov entropy. An interesting question is to determine its regularity. More precisely, a phase transition is a point $\beta_0$ where the pressure function is not analytic. Though a restrictive property, the regularity of the potential can transfer to the pressure function. It is known for example, that in the case of symbolic dynamics, when our dynamical system is a subshift of finite type, having a Hölder potential implies that the pressure function is analytic. Such phase transitions and objects have been studied in \cite{bowen, coronel_low_temp,ruelle,sarig} for instance.

In our case we are interested in the fullshift on unilateral infinite sequences $(\mathcal{A}^\N, \sigma)$ where $\mathcal{A}$ is a finite alphabet, and we define our potential depending on a substitution. This potential was already studied in \cite{bruin_leplaideur13} for Thue-Morse substitution,  in \cite{bruin_leplaideur15} Fibonacci substitution, and then  in \cite{bhl} for a large class of substitutions (which includes Thue-Morse but not Fibonacci). Our main focus is the renormalization of this potential for $k$-bonacci substitutions. The means to that end are essentially word combinatorics on the full-shift and on the subshift associated to the substitution. The renormalization operator was introduced in \cite{baraviera_renormalization} for constant length substitution, adapted to Fibonacci substitution in \cite{bruin_leplaideur15}  and further to non constant length substitution in \cite{bhl}. It was introduced for its expected links with phase transition \--- which we define later\--- but such links are not understood at the moment. 

We are also interested in freezing phase transition, which is a critical $\beta$ after which the pressure function is affine. We show that we can apply a criterion from \cite{bhl} in the case of $k$-bonacci substitutions in order to get freezing phase transitions. 

This paper is inspired by \cite{bhl}, in which the authors treat the case of the existence of freezing phase transitions and of the renormalization of potentials for a certain class of substitutions: $2$-full marked primitive aperiodic substitutions. We recall that a substitution is $2$-full if every word of length $2$ is in its language and marked if both the set of first and last letters of the images of letters by the substitution are the whole alphabet. We adapt their techniques to a family of substitutions that are neither marked nor 2-full. Indeed, $k$-bonacci substitutions not being marked changes the study of the renormalization operator, in particular its first step which consists of Proposition \ref{p.kbo_prefix}. This involves a combinatorial study of $k$-bonacci substitutions, and later, a precise understanding of the 'recognizability' of these substitutions which is stronger than in the cases studied in \cite{bhl} and implies a simpler expression for the expression of the renormalization operator.

%
%

\subsection{Results}

Let $\mathcal{A}$ be a finite set called an alphabet. In particular, let us denote $\A_k:=\left\{0,1,..,k-1\right\}$ where $k$ is a positive integer. $\mathcal{A}^*$ denotes the free monoid generated by $\mathcal{A}$ by concatenation. In other words, it is the set of finite words over the alphabet $\mathcal{A}$ with the operation of concatenation. We denote by $\varepsilon$ the empty word.
 
 For any word $w$ in $\mathcal{A}^*$, $|w|$ is the length of the word $w$. For any letter $a$ in $\mathcal{A}$, $|w|_{a}$ denotes the number of occurrences of $a$ in $w$.
 
  $\mathcal{A}^\N$ denotes the set of right-handed sequences over the alphabet $\mathcal{A}$. An element of $\mathcal{A}^\N$ is called a configuration. This set is endowed with the product topology which is compatible with the distance $d$ defined by $d(x,y)=\frac{1}{2^{\min\{n\in \N | x_n \neq y_n\}}}$ whenever $x\neq y$.
  
  Let $u=u_0...u_n$ and $v=v_0...v_m$ be two words in $\mathcal{A}^*$. If there exists a non negative integer $k$ such that $u=v_k...v_{k+n-1}$ then we say that $u$ is a subword of $v$ and we write $u \sqsubset v$.
  Let $x \in \mathcal{A}^\N$ and $n$ and $m$ be two positive integers. We denote by $x_{[n...n+m]}$ the word $x_nx_{n+1}...x_{n+m}$.

A substitution is a non erasing morphism of the free monoïd $\A^*$. Let us now introduce, for any integer $k$, the $k$-bonacci substitution. 

\begin{defi}\label{def.kbo}
  For any integer $k\geq1$, we define the $k$-bonacci substitution $s_k: \mathcal{A}_k^* \to \mathcal{A}_k^*$ on the generators of $\mathcal{A}_k^*$ in the following way:
  $$
  \forall a \in \mathcal{A}_k\backslash\{k-1\},\  s_k(a)=0(a+1)
  $$
  and
  $$
  s_k(k-1)=0 
  $$
\end{defi}

Then, for a continuous function $V: \A^\N \rightarrow \R_+$ called the potential, we defined the renormalization operator introduced in \cite{baraviera_renormalization}:
 
$$
\begin{array}{ccccc}
  R&:&\mathcal{C}(\mathcal{A}_k^\N,\R)&\rightarrow&\mathcal{C}(\mathcal{A}_k^\N,\R)\\
   & & V(x)                              &\mapsto & \displaystyle{\sum_{j=0}^{|s_k(x_0)|-1} V\circ \sigma^j \circ s_k(x)}.
 \end{array}
$$

Let us also define the attractor $\Sigma_{s_k}$ in the following way:

let $\omega=\lim_{+\infty}s_k^n(0)$, $\Sigma_{s_k}= \overline{\text{Orb}(\omega)}$, the set ${\text{Orb}(\omega)}$ being the orbit of the point $\omega$ under the action of the shift.

Then, for $x \in \A_k^\N \backslash \Sigma_{s_k}$, define $\delta_{s_k}(x)$ with the following identity $d(x,\Sigma_{s_k})=1/2^{\delta_{s_k}(x)}$.
If $x \in \Sigma_{s_k}$, $\delta_{s_k}(x)=+\infty$.

Our first theorem deals with the existence of a fixed point for this renormalization operator.
\begin{thm}\label{t.renorm}
 Let $k \geq 2$, there exists $U \in \mathcal{C}\left(\mathcal{A}_k^\N,\R_+\right)$ such that $RU=U$ given by:
  $$
 \forall x \in \mathcal{A}_k^\N,\, U(x)=\log\left(1+\frac{v_{x_0}}{\frac{\lambda}{\lambda-1}+\sum_{l\in\mathcal{A}_k} v_l |x_{[0..\delta_{s_k}(x)-1]}|_l -v_{x_0}}\right)
 $$
 where $\lambda$ is the dominant root of the polynomial $X^k-\displaystyle\sum_{j=0}^{k-1}X^j$ and 
  $$
 \forall l \in \mathcal{A}_k,\, v_l=\frac{1}{\lambda^{k-1-l}}\sum_{j=0}^{k-1-l}\lambda^j.
 $$
 Moreover, if $V:\mathcal{A}_k^\N\rightarrow\R_+$ is of the form:
 $$
\forall x \in \mathcal{A}_k^\N,\, V(x)=\frac{g(x)}{\delta_{s_k}(x)^\alpha}+\frac{h(x)}{\delta_{s_k}(x)^\alpha}
$$
with $g$ a positive continuous function, $h$ being $0$ on $\Sigma_{s_k}$ and continuous and $\alpha >0$, then, for any $x$ in $\mathcal{A}_k^\N\backslash\Sigma_{s_k}$:

$$
 \lim_{n\rightarrow +\infty} R^n V (x)=\left\{\begin{array}{ll}
                                       0                              &\text{if}\; \alpha>1\\
                                       +\infty                        &\text{if}\; \alpha<1\\
                                       \int\, g\, d\mu\, \cdot \, U(x)  &\text{if}\; \alpha=1
                                              \end{array}\right.
 $$
 where $\mu$ is the only ergodic probability measure on $\Sigma_{s_k}$.
\end{thm}

In particular, let us give the expression of the fixed point $U$ of the renormalization operator for the case $k=3$. Let $\lambda$ be the dominant root of the Tribonacci polynomial $X^3-X^2-X-1$. This root can be computed via Cardan's method and is
$$\lambda=\frac{\sqrt[3]{19+3\sqrt{33}}+\sqrt[3]{19-3\sqrt{33}}+1}{3}.$$

$U$ is given by~:

$$
\forall x \in \A_3^\N,\ U(x)=\left\{\begin{array}{lcr}
                                     \log\left(1+\frac{\lambda}{\frac{\lambda}{\lambda-1}+\lambda|w|_0+\frac{\lambda+1}{\lambda}|w|_1 + |w|_2+\lambda}\right)&\text{if}&x_0=0\\
                                     \log\left(1+\frac{\frac{\lambda+1}{\lambda}}{\frac{\lambda}{\lambda-1}+\lambda|w|_0+\frac{\lambda+1}{\lambda}|w|_1 + |w|_2+\frac{\lambda+1}{\lambda}}\right)&\text{if}&x_0=1\\
                                     \log\left(1+\frac{1}{\frac{\lambda}{\lambda-1}+\lambda|w|_0+\frac{\lambda+1}{\lambda}|w|_1 + |w|_2+1}\right)&\text{if}&x_0=2
                                    \end{array}\right.
$$
where $w$ is the word $x_{[0..\delta_{s_{3}}(x)-1]}$.

Our second theorem deals with the existence of freezing phase transition for a potential in the same family as in Theorem \ref{t.renorm}.
\begin{thm}\label{t.fpt}
 For any integer $k \geq 2$, for any potential $V:\A_k^\N\rightarrow\R_+$ of the form $V(x)=\frac{g(x)}{\delta_{s_k}(x)}+\frac{h(x)}{\delta_{s_k}(x)}$, then there exists a real number $\beta_c$ such that:
 \begin{itemize}
  \item $P$ is analytic on $[0,\beta_c)$: there is a unique equilibrium state which has full support for every $\beta \in [0,\beta_c)$.
  \item $P$ is identically zero on $(\beta_c,0]$ and the unique ergodic measure of $(\Sigma_{s_k},\sigma)$ is the unique equilibrium state.
 \end{itemize}
 Such a point $ \beta_c$  is called a freezing phase transition.
\end{thm}

We recall that, in our settings, a Hölder regularity of the potential would imply the analyticity of the pressure function. We are interested in a potential supported outside the attractor of the substitution and which is not Hölder continuous. The family of potentials on which the previous theorems holds is a large class of examples which satisfy these conditions.

\subsection{Outline of the paper}

The general idea is that a substitution generates an attractor in the fullshift (the subshift associated to the substitution). We use the attractor to define a family of potentials on the full shift supported outside the attractor. Then the substitution can generate a renormalization operator for these potentials. Iterating the renormalization over a potential makes it converge towards a fixed point for the renormalization that is still a potential supported outside the attractor.

Then, for any potential in this family, the pressure function has a freezing phase transition and the mesure that realizes the supremum is the only ergodic measure supported on the attractor

The fact that $k$-bonacci substitutions are not left marked gives a different behaviour for the renormalization operator as points of the fullshift tend to converge faster, making, in that precise case, useless the 2-full hypothesis. However we cannot use the property of being marked and so the following techniques also adapt ideas from \cite{mosse_reco}.

Section \ref{s.substitutions} introduces the basic definitions for substitutions and their languages, as well as some classical properties of these objects. 

Then in Section \ref{s.kbo} we prove some crucial properties for our study regarding $k$-bonacci substitutions. We give an explicit formula for the speed of convergence towards the subshift, and further study the effect of applying the shift a ``small'' number of times. This is very important because it allows us to compute the fixed point for the renormalization operator of Theorem \ref{t.renorm}.
  
Section \ref{s.potential} is dedicated to the proof of Theorem \ref{t.renorm}. In the first subsection we prove a weaker version of the theorem ($\alpha=1$, $g \equiv 1$, $h\equiv 0$). The previous study of combinatorial properties allows us to compute the fixed point with sums very close to being Riemann sum, close enough that the same techniques work. The second subsections finishes the proof for the strong version of the theorem.

Finally, in Section \ref{s.fpt} we prove the existence of a freezing phase transition by applying a criterion given in \cite{bhl} (Theorem 3).

In order to help the reader understanding, notations being quite involved for some propositions, some proofs for the particular example of the Tribonacci substitution ($k=3$) are given at the end of the paper in Appendix \ref{s.tribo}. Almost every argument used in the involved proofs of the general case are already used in the proofs for Tribonacci hence this example is almost sufficient to understand the general case.

\section{Generalities on substitutions}\label{s.substitutions}

Let us recall the definition of the main object in our article:

\begin{defi}
 A substitution over the alphabet $\mathcal{A}$ is a non-erasing morphism over the monoid $\mathcal{A}^*$. 
\end{defi}

\begin{ex}
 The Fibonacci substitution defined over $\{0,1\}^*$ given by:
 \begin{align*}
& 0 \mapsto 01 \\
 &1 \mapsto 0
 \end{align*}
\end{ex}

Given a substitution $s$ over a finite alphabet $\mathcal{A}$, we can define the language of $s$:

\begin{defi}
 Let $\mathcal{A}$ be a finite alphabet and $s$ a substitution on this alphabet. The language of $s$ is:
$$
\lang_s=\{w \in \mathcal{A}^*\ | \  \exists a \in \mathcal{A}, \  \exists k \in \N, \  w \sqsubset s^k(a)\}
$$
(where $w \sqsubset s^k(a)$ means that $w$ is a subword of $s^k(a)$).
\end{defi}

\begin{rem}\label{r.factlang}
 The language $\lang_s$ of a substitution is factorial i.e. any subword of a word in $\lang_s$ is in $\lang_s$. It is also extendable i.e. for any word $w \in \lang_s$ there exists a pair $(a,b) \in \A$ such that $awb \in \lang_s$.
\end{rem}

The language of a substitution also has some special words:
\begin{defi}\label{d.bispecial}
 A word $w \in \lang_s$ is right-special if there exists two distinct letters $a$ and $b$ such that both $wa$ and $wb$ are in the language. 
  A word $w \in \lang_s$ is left-special if there exists two distinct letters $a$ and $b$ such that both $aw$ and $bw$ are in the language.  
  A word is bispecial if it is both left-special and right-special.
\end{defi}

This notion was first introduced in \cite{cassaigne_complexite} in order to study the complexity of a language (that is to say the function which counts the number of words of any given size in the language).

\begin{defi}
 Let $s$ be a substitution over an alphabet $\mathcal{A}=\{0,...,k-1\}$. We call incidence matrix the matrix $S \in \mathcal{M}_k(\N)$ defined by:
 $$
 \forall (i,j)\in\llbracket0,k-1\rrbracket^2, \  S_{i,j}=|s(j)|_{i}.
 $$
\end{defi}

\begin{defi}
 We say that a substitution is primitive if its incidence matrix is primitive. Namely:
 $$
 \exists n \in \N,\ \forall (i,j)\in\llbracket0,k-1\rrbracket^2, \  \left(S^n\right)_{i,j} >0.
 $$
\end{defi}

\begin{rem}
The point of asking for a primitive substitution is, mostly, to take advantage of linear algebra and more specifically Perron Frobenius theorem.
\end{rem}

\begin{defi}
Let $s$ be a substitution over the alphabet $\mathcal{A}$. 

If $\{s(a)_0\,|\, a \in \mathcal{A}\}=\mathcal{A}$, then $s$ is said to be left-marked.

If $\{s(a)_{|s(a)|-1}\,|\, a \in \mathcal{A}\}=\mathcal{A}$, then $s$ is said to be right-marked.

 A substitution which is both left-marked and right-marked is marked.
\end{defi}

 In all that follows, we study a family of primitive right-marked substitution.
 
 \begin{rem}
  Let $s$ be a primitive substitution over a finite alphabet $\mathcal{A}$ and assume that there exists a letter $a$ in $\mathcal{A}$ such that $s(a)$ starts by $a$ (which is always the case up to taking a power of $s$).
  
  Then the sequence $(s^n(a))_{n \in \N}$ converges to a right handed sequence in $\mathcal{A}^\N$ which is a fixed point for $s$ (if we extend the definition of $s$ to the set $\mathcal{A}^\N$).
  
  If a substitution $s$ admits only one fixed point, we note it $\omega^s$.
 \end{rem}
 
 \begin{defi}\cite{mosse_reco}\label{d.circularpoint}
 Let us define, for a given substitution $s$ with a unique fixed point $\omega^s$ the following family of sets:
 $$
 \forall n \in \N^*, D^n_s=\left\{\left|s^n(\omega^s_{[0...k]})\right|, \  k \in \N\right\} \cup \{0\}.
 $$
\end{defi}

\begin{rem}\label{r.circularpoint}
 It is obvious to check the following assertion for any substitution $s$ with a unique fixed point:
 $$
 \forall n \in \N^*, D^{n+1}_s \subset D^n_s.
 $$
\end{rem}
 
 \begin{defi}
  Let $s$ be a substitution over the alphabet $\mathcal{A}$. We define the subshift $\Sigma_s \subset \mathcal{A}^\N$ associated to the substitution $s$ by:
  $$
  \Sigma_s:= \{x \in \mathcal{A}^\N \  | \forall w \in \mathcal{A}^*,\  w \sqsubset x \implies w \in \lang_s\}.
  $$
 \end{defi}

We are usually interested in primitive substitutions which admit a non periodic fixed point. In these cases we generally study the dynamical system $(\Sigma_s,\sigma)$ where $\sigma$ is the shift on the set of right handed sequences. We recall that the shift $\sigma$ is defined by:
  $$
  \begin{array}{ccccc}
   \sigma &: & \mathcal{A}^\N & \rightarrow &\mathcal{A}^\N \\
          &   & x_0x_1...x_n...    & \mapsto     & x_1x_2...x_{n+1}...
  \end{array}
  $$
  
   \begin{rem}
For any substitution $s$ which admits a non ultimately periodic fixed point under the action of the shift, $\Sigma_s$ is a Cantor set for the product topology on $\mathcal{A}^\N.$
 \end{rem}
 
  We also recall the following theorem.
   \begin{theorem}[\cite{queffelec}]\label{t.ergo}
  If $s$ is a primitive substitution, then the dynamical system $(\Sigma_s,\sigma)$ is minimal, uniquely ergodic and has topological and Kolmogorov entropy 0.
 \end{theorem}

  In this paper however, we are interested in the action of a particular substitution (chosen amongst a family of substitutions) on points in $\mathcal{A}^\N$. Namely, we wish to know how fast the orbit of a sequence in $\mathcal{A}^\N$ gets close to the compact Cantor set $\Sigma_s$.
  
  In order to do that, we introduce the following object:
  
  \begin{defi}
 Let $\mathcal{A}$ be a finite alphabet and let $s$ be a primitive substitution on this alphabet. Let us define the following function on $\mathcal{A}^\N$:
 $$
 \begin{array}{ccccc}
   \delta_s  &:  &  \mathcal{A}^\N  &  \to      & \N\cup\{+\infty\} \\
             &    &     x            &  \mapsto  & \sup \{n \in \N \  | \  x_{[0...n-1]} \in \lang_s\}.
 \end{array}
 $$
\end{defi}

 We are interested, for any primitive substitution $s$ on $\mathcal{A}$, $x \in \mathcal{A}^\N$ and $k>0$, in $\delta_s(s^k(x))$. Such a quantity is interesting because we have the following equality:
 
 $$
 \forall k \in \N, \  \forall x \in \mathcal{A}^\N, \  \text{dist}(s^k(x),\Sigma_s)=\frac{1}{2^{\delta_s(s^k(x))}}.
 $$
 
 Hence, knowing precisely the behaviour of the sequence $\left( \delta_s s^n(x) \right)_{n\in \N}$ gives precisely the sequence of distances between $s^n(x)$ and the subshift $\Sigma_s$.

\section{The $k$-bonacci substitution}\label{s.kbo}
  
In this section, we denote by $\mathcal{A}_k$ the set $\{0,...,k-1\}$ for any integer $k\geq2$.  We recall that the $k$-bonacci substitution is given by its images on the generators of by:

  $$
  \forall a \in \mathcal{A}_k\backslash\{k-1\},\  s_k(a)=0(a+1)
  $$
  and
  $$
  s_k(k-1)=0 
  $$

In the particular case of $s_3$, the substitution is well know as `Tribonacci substitution'. This substitution was first introduced by Rauzy in \cite{rauzy} and further studied in \cite{arnoux_rauzy}. It was the first example of a substitution studied for its underlying geometric properties (which introduced Rauzy fractals), though these geometric properties do not play any role in the present paper. The case of $k$-bonacci substitutions is an analogous and they are notable for their complexity function: there is exactly $kn+1$ words of length $n$ in the language of the $k$-bonacci substitution.

Remark that these substitution define uniquely ergodic dynamical systems with zero entropy from Theorem \ref{t.ergo}.

\begin{rem}\label{r.reccur_rel}
 For any integer $k\geq2$ we have the following relation:
 $$
 \forall n \in \N, \  s_k^{n+k}(0)=s_k^{n+k-1}(0)s_k^{n+k-2}(0)...s_k^n(0).
 $$
\end{rem}

Let us prove the following proposition:

\begin{prop}\label{p.kbo_prefix}
  For any integer $k \geq 2$, for any sequence $x \in \mathcal{A}_k^\N$ such that $\delta_{s_k}(x)=p \in \N$ and for any positive integer $n$, the maximal prefix of $s_k^n(x)$ in $\lang_{s_k}$ is:
  $$
  s_k^n(x_{[0...p-1]})s_k^{n-1}(0)...s_k(0)0.
  $$
\end{prop}

  We recall that a proof of this proposition is provided for the case $k=3$ in the appendix.

\begin{proof}
 Let $k$ be an integer greater than $1$ and $x$ a sequence in $\mathcal{A}^\N$ and let us note $p=\delta_{s_k}(x)$.
 
 We have $x_{[0...p-1]} \in \lang_{s_k}$ but $x_{[0...p]} \notin \lang_{s_k}$. As a consequence there exists a letter $a$ in $\mathcal{A}_k$ different from $x_p$ such that $x_{[0...p-1]}a \in \lang_{s_k}$. Hence $s_k(x_{[0...p-1]}a) \in \lang_{s_k}$ which implies that $s_k(x_{[0...p-1]})0 \in \lang_{s_k}$ since evey image of letter starts with $0$. Moreover, the word $s_k(x_{[0...p-1]})0$ is a prefix of $s_k(x)$. It is in fact the maximal prefix in $\lang_{s_k}$ because a longer prefix in $\lang_{s_k}$ has a preimage containing $x_{[0...p]}$ which would be in $\lang_{s_k}$ which is false. 
 
 Iterating the substitution $n$ times ends the proof.
\end{proof}

 Let us state an immediate corollary:
 
 \begin{cor}
  For any integer $k \geq 2$, for any sequence $x \in \mathcal{A}_k^\N \backslash \Sigma_{s_k}$ such that $\delta_{s_k}(x)=p$ and for any positive integer $n$ we have:
  $$
  \delta_{s_k}(s_k^n(x))=\sum_{j\in \mathcal{A}}|s_k^n(j)||x_{[0...p-1]}|_j+\sum_{l=0}^{n-1}|s_k^l(0)|.
  $$
 \end{cor}

 Let us now study the effect of the shift action on the function $\delta_{s_k}$ for any integer $k$ greater than $1$.
 
 First, we prove the following proposition:
 
 \begin{prop}\label{p.recog_kbo}
  For any integer $k$ greater than $1$, for any integer $n\geq k$, for any non negative integer $d$,
  $$
  \omega^{s_k}_{[d,...,d+|s_k^n(0)|-1]}=s_k^n(0)\Leftrightarrow d\in D_{s_k}^n,
  $$
  where $\omega^{s_k}$ denotes the unique fixed point of the $k$-bonacci substitution.
 \end{prop}

 \begin{rem}
  This proposition is an efficient formulation Theorem 3.1bis of \cite{mosse_reco} for the specific case of the powers of $k$-bonacci substitution. It gives a bound for the recognizability constant and states which words mark points in $D_n$. Indeed it states that, for any $n \geq k$, the only words of length $|s_k^n(0)|$ that can be seen as an $n^{th}$ image of letter in the fixed point of $s_k$ always appear as prefixes of $|s_k^n(0)|$. This allows to ``cut'' the fixed point of $s_k$ in blocks which are $n^{th}$ images of letters by looking only at words  of length $|s_k^n(0)|$ which makes the constant of Theorem 3.1bis of \cite{mosse_reco} completely explicit. This is however completely dependant on the choice of $k$-bonacci substitution.
 \end{rem}

  We recall that a proof of this proposition is provided for the case $k=3$ in the appendix.

 \begin{proof}
  First, let us prove that for any integer $k$ greater than $1$, for any integer $n\geq k$ and for any integer $d$,
  $$
  \omega^{s_k}_{[d,...,d+|s_k^n(0)|-1]}=s_k^n(0)\Leftarrow d\in D_{s_k}^n.
  $$
  If $d$ is in $D_{s_k}^n$ then we have two cases to treat:
  \begin{itemize}
   \item Either $\omega^{s_k}_{[d,...,d+|s_k^n(0)|-1]}=s_k^n(0)$ in which case the equality is trivialy verified.
   \item Or there is a letter $a\in \mathcal{A}_k\backslash \{0\}$ such that $\omega^{s_k}_{[d,...,d+|s_k^n(0)|-1]}$ is prefix of $s_k^n(a)s_k^n(0)$. Let us notice that:
   $$
   s_k^n(a)=s_k^{n-1}(0)...s_k^{n-(k-a)}(0).
   $$
   Moreover, $s_k^{n-(k-a)-1}(0)...s_k^{n-k}(0)$ is a prefix of $s_k^n(0)$. Indeed the word $s_k^{n-(k-a)}(0)$ is a prefix of $s_k^n(0)$. This is because $0$ is a prefix of $s_k(0)$.
   Using Remark \ref{r.reccur_rel}, we have that $s_k^n(0)$ is a prefix of $s_k^n(a)s_k^n(0)$. Hence:
   $$
   \omega^{s_k}_{[d,...,d+|s_k^n(0)|-1]}=s_k^n(0).
   $$
   \end{itemize}

   Let us now prove the converse statement:
   
   For any integer $k$ greater than $1$, for any integer $n\geq k$ and for any integer $d$,
  $$
  \omega^{s_k}_{[d,...,d+|s_k^n(0)|-1]}=s_k^n(0)\Rightarrow d\in D_{s_k}^n.
  $$
  We prove this implication by induction on $n$.
  
  \textit{Claim:}  
  The property holds for $n=k$. Namely, every occurrence of $s_k^k(0)$ in $\omega^{s_k}$ has starting position in $D_{s_k}^k$.
  First, let us state that:
  $$
  s_k^k(0)=s_{k-1}^{k-1}(0)(k-1)s_{k-1}^{k-1}(0)
  $$
  since
  $$
  s_k^k(0)=s_{k}^{k-1}(0)s_{k}^{k-1}(1).
  $$
  It is then easy to check that 
  $$
  s_{k}^{k-1}(0)=s_{k-1}^{k-1}(0)(k-1)
  $$
  and
  $$
  s_{k}^{k-1}(1)=s_{k-1}^{k-1}(0).
  $$
  
  Let us also remark that for any $a \in \mathcal{A}_k$, $s_k^k(a)$ starts with $ s_{k}^{k-1}(0)$. Hence $s_{k-1}^{k-1}(0)(k-1)$ is prefix of every $k^{\textit{th}}$ image of letter.
  Moreover, it is clear from the equality $s_k^k(0)=s_{k-1}^{k-1}(0)(k-1)s_{k-1}^{k-1}(0)$ that every $k^{\textit{th}}$ image of letter contains one and only one occurrence of the letter $k-1$ since $s_{k-1}$ is defined on $\{0,...,k-2\}^*$.
  
  Finally, $s_k^k(0)$ being the image with maximal length, every occurrence of this word has to contain at least one point in $D_{s_k}^k$. The previous remarks imply that we only see this word in $\omega^{s_k}$ starting from a point in $D_{s_k}^k$, thus completing the initialisation of the induction.
  
  Let us now assume that this property is true for a fixed integer $n\geq k$. We wish to prove that it is still true for $n+1$. We write $s_k^{n+1}(0)$ in the following way:
  \bigskip
  \begin{center}
     \begin{tikzpicture}
 \draw (0,0)--(8,0);
 \draw (0,0.2)--(0,-0.2);
 \draw (5,0.2)--(5,-0.2);
 \draw (8,0.2)--(8,-0.2);
 \draw (2.5,0) node [above] {$s_k^n(0)$};
 \draw (6.5,0) node [above] {$s_k^n(1)$};
 \draw (0,-0.2) node [below] {$d$};
 \draw (5,-0.2) node [below] {$e$};
\end{tikzpicture}
\end{center}

By induction hypothesis, $d \in D_{s_k}^n$. Let us now assume that $e$ is also in $D_{s_k}^n$. Then $d$ and $e$ are necessarily two consecutive points of $D_{s_k}^n$. Since images of letters by $s_k$ are of length one or two, two consecutive points of $D_{s_k}^n$ have either one of them or two of them in $D_{s_k}^{n+1}$. Hence $d \in D_{s_k}^{n+1}$ or $e \in D_{s_k}^{n+1}$. Notice that if $e \in D_{s_k}^{n+1}$, writing $s_k^n(0)=s_k^{n+1}(k-1)$ is enough to prove that $d \in D_{s_k}^{n+1}$.

If however $e$ is not in $D_{s_k}^n$, then there exists a letter $a$ in $\mathcal{A}_k$ different from $0$ such that we can write $s_k^{n+1}(0)$ in the following way:
$$
s_k^{n+1}(0)=s_k^{n}(a)ws_k^n(1).
$$

We represent this writing on the following picture:
   
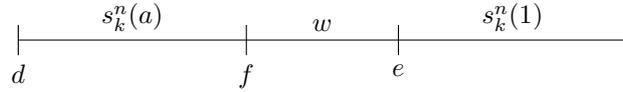
\begin{center}
     \begin{tikzpicture}
 \draw (0,0)--(8,0);
 \draw (0,0.2)--(0,-0.2);
 \draw (3,0.2)--(3,-0.2);
 \draw (5,0.2)--(5,-0.2);
 \draw (8,0.2)--(8,-0.2);
 \draw (1.5,0) node [above] {$s_k^n(a)$};
 \draw (4,0) node [above] {$w$};
 \draw (6.5,0) node [above] {$s_k^n(1)$};
 \draw (0,-0.2) node [below] {$d$};
  \draw (3,-0.2) node [below] {$f$};
 \draw (5,-0.2) node [below] {$e$};
\end{tikzpicture}
  \end{center}
  with $f$ in $D_{s_k}^n$. Moreover, $s_k^n(0)=s_k^{n-1}(0)...s_k^{n-k}(0)$ and it is easily seen that $s_k^n(a)=s_k^{n-1}(0)...s_k^{n-(k-a)}(0)$ so we deduce that $w=s_k^{n-(k-a)-1}(0)...s_k^{n-k}(0)$. 
  Now let us remark that $s_k^{n-1}(0)...s_k^{n-(k-a)}(0)$ is prefix of $s_k^n(1)$ whenever $a\neq0$ and $n\geq k$.
  
  Since $f$ is in $D_{s_k}^n$, we should have 
  $$
  s_k^{n-(k-a)-1}(0)...s_k^{n-k}(0)s_k^{n-1}(0)...s_k^{n-(k-a)}(0)=s_k^n(0)
  $$
  using the converse implication of Property \ref{p.recog_kbo} which is already proved. This is not possible because the last letter is different, since $n-(k-a) \not\equiv n [k]$ if $a\in \mathcal{A}_k\backslash\{0\}$.
  
  The only possible case being $e \in D_{s_k}^n$ which necessarily implies that $d \in D_{s_k}^{n+1}$, the proof is complete.
 \end{proof}
 
 We can now prove the following property:
 \begin{prop}\label{thm.length_kbo}
   For any integer $k\geq2$, for any sequence $x \in \mathcal{A}^\N$ such that $\delta_{s_k}(x)=p$, we have:
   $$
   \forall n\geq k, \  \forall j<|s_k^n(x_0)|, \  \delta_{s_k}(\sigma^j(s_k^n(x))=\sum_{l\in \mathcal{A}}|s_k^n(l)||x_{[0...p-1]}|_j+\sum_{l=0}^{n-1}|s_k^l(0)|-j.
   $$
 \end{prop}

 \begin{proof}
   First we recall that the maximal prefix of $s_k^n(x)$ in the language is $s_k^n(x_{[0...p-1]})s_k^{n-1}(0)...0$ if $\delta_{s_k}(x)=p$ (see Proposition \ref{p.kbo_prefix}). 
 
 Moreover, it is obvious that: 
 $$
 \forall n \geq k, \  \forall j < |s_k^n(x_0)|, \  \delta_{s_k}(\sigma^js_k^n(x))\geq\delta_{s_k}(s_k^n(x))-j.
 $$
 
 Moreover we remind that, since $j < |s_k^n(x_0)|$, the action of the shift does not completely erase the image of the first letter. Let us represent $\sigma^js_k^n(x)$ by a semiline:
 
 \bigskip
  
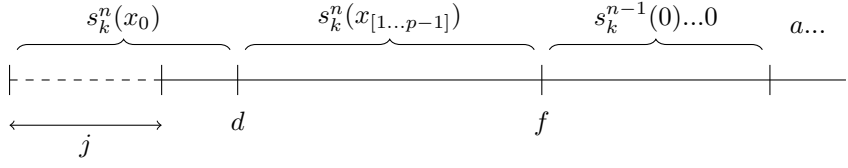
\begin{center}
     \begin{tikzpicture}
 \draw [dashed] (0,0)--(2,0);
 \draw  (2,0)--(11,0);
 \draw (0,0.2)--(0,-0.2);
 \draw (2,0.2)--(2,-0.2);
 \draw (3,0.2)--(3,-0.2);
 \draw (7,0.2)--(7,-0.2);
 \draw (10,0.2)--(10,-0.2);
 \draw[decorate,decoration={brace, amplitude=4pt, raise=0.4cm}](0.1,0)--(2.9,0);
 \draw[decorate,decoration={brace, amplitude=4pt, raise=0.4cm}](3.1,0)--(6.9,0); 
 \draw[decorate,decoration={brace, amplitude=4pt, raise=0.4cm}](7.1,0)--(9.9,0);
 \draw [<->] (0,-.6)--(2,-.6);
 \draw (1,-.6) node [below] {$j$};
 \draw (5,.5) node [above] {$s_k^n(x_{[1...p-1]})$};
 \draw (8.5,.5) node [above] {$s_k^{n-1}(0)...0$};
 \draw (1.5,.5) node [above] {$s_k^{n}(x_0)$};
 \draw (10.5,.5) node [above] {$a...$};
 \draw (3,-.3) node [below] {$d$};
 \draw (7,-.3) node [below] {$f$};
\end{tikzpicture}
  \end{center}
  where $a \in \mathcal{A}_k$, and $d$ and $f$ are a pair of integers such that
$$
\omega^{s_k}_{[d..d+f-1]}\omega^{s_k}_{[f...f+|s_k^{n-1}(0)...0|-1]}=s_k^n(x_{[1...p-1]})s_k^{n-1}(0)...0.
$$
It is enough to remark that for any $n\geq k$ we have:
$$
s_k^{n-1}(0)...0=s_k^{n-1}(0)...s_k^{n-k}(0)s_{k}^{n-k-1}(0)...0
$$
and recall that 
$$
s_k^{n-1}(0)...s_k^{n-k}(0)=s_k^{n}(0)
$$
so we have $s_k^{n-1}(0)...0=s_k^{n}(0)s_k^{n-k-1}(0)...0$ and we can deduce from Proposition \ref{p.recog_kbo} that $f$ is in $D_{s_k}^n$ (regardless of the choice of $f$).

Let $u_0...u_l:=s_k^n(x_{[0...p-1]})s_k^{n-1}(0)...0$. Let $w=u_j...u_l$.
Let us assume that $wa \in \lang_{s_k}$ which is equivalent to the assumption that $\delta_{s_k}(\sigma^js_k^n(x))>\delta_{s_k}(s_k^n(x))-j$.

The $k$-bonacci substitution being right-marked, and since $f$ is in $D^n$, it is easy to remark that $d$ is also in $D^n$. Then, since $j<|s_k^n(x_0)|$, one can always read on the left of $d$ the last letter of $s_k^n(x_0)$. The substitution $s_k$ (and thus $s_k^n$) being right-marked, we conclude that every occurrence of $w$ in $\omega$ is as a subword of $s_k^n(x_{[0...p-1]})s_k^{n-1}(0)...0$. Hence if $wa$ were in $\lang_{s_k}$, so would $s_k^n(x_{[0...p-1]})s_k^{n-1}(0)...0a$ which is a contradiction with Proposition \ref{p.kbo_prefix}.
 \end{proof}

\section{Fixed point of the renormalization operator}\label{s.potential}

\subsection{For a potential of the form $\frac{1}{\delta_{s_k}(x)}$}

In all this section, $k$ denotes a fixed integer greater than 2. We remind that the renormalization operator $R$ is defined by:

$$
\begin{array}{ccccc}
  R&:&\mathcal{C}(\mathcal{A}_k^\N,\R)&\rightarrow&\mathcal{C}(\mathcal{A}_k^\N,\R)\\
   & & V(x)                              &\mapsto & \displaystyle\sum_{j=0}^{|s_k(x_0)|-1} V\circ \sigma^j \circ s_k(x)
 \end{array}
$$

\begin{lemma}\label{l.ruelle_power}
 For any integer $n$ and for any configuration $x$ in $\mathcal{A}_k^\N$:
 $$
 R^n(V)(x)=\sum_{j=0}^{|s_k^n(x_0)|-1} V\circ \sigma^j \circ s_k^n(x)
 $$

\end{lemma}

 This is proved by induction in \cite{bhl} for any substitution.

 First, we wish to understand the asymptotic behaviour of $R^n(V)$ as $n$ goes to infinity for the potential $V_0$ defined by:
 $$
 \begin{array}{ccccc}
  V_0 &: & \mathcal{A}_k^\N & \rightarrow & \R \\
    &   & x              & \mapsto     & \frac{1}{\delta_{s_k}(x)}.
 \end{array}
$$

\begin{rem}\label{r.perron}
 From Perron Frobenius Theorem, we know that for any letter $l$ in $\mathcal{A}_k$, there exists a real sequence $r_l(n)$ and a positive real number $\gamma_l$ such that $|s_k^n(l)|= \gamma_l\lambda^n +r_l(n)$ where $\lambda$ is the Perron Frobenius eigenvalue of the incidence matrix of $s_k$ and $r_l(n)$ satisfies:
 $$
 \forall n \in \N, \  |r_l(n)|\leq C_l\theta^n
 $$
 with $C_l>0$ and $0<\theta<\lambda$.
\end{rem}

Notice that $\lambda$ is the dominant root of the polynomial $X^k-\displaystyle\sum_{j=0}^{k-1}X^j$.
 
 \begin{rem}\label{r.eigenvector}
   Remark also that the vector $(\gamma_0,...,\gamma_{k-1})$ is a multiple of the left eigenvector $v=(v_0,...,v_{k-1})$ associated to the eigenvalue $\lambda$.
 We can write:
 $$
 \forall l \in \mathcal{A},\, v_l=\frac{1}{\lambda^{k-1-l}}\sum_{j=0}^{k-1-l}\lambda^j.
 $$
 \end{rem}

 From this formula, notice that we chose an eigenvector such that $v_0=\lambda$

We prove the following:

\begin{thm}\label{t.renorm_baby_case}
For any $k\in \N$, $k\geq2$, there exists $U \in C(\mathcal{A}_k^\N,\R_+)$, fixed point for the  renormalization operator associated to the $k$-bonacci substitution, defined by:
 $$
 \forall x \in \mathcal{A}_k^\N,\, U(x)=\log\left(1+\frac{v_{x_0}}{\frac{\lambda}{\lambda-1}+\sum_{l\in\mathcal{A}_k} v_l |x_{[0..\delta_{s_k}(x)-1]}|_l -v_{x_0}}\right).
 $$

 Moreover, for any configuration $x \in \mathcal{A}_k^\N$ we have:
 $$
 \lim_{n\rightarrow + \infty} R^n V_0 (x) = U(x).
 $$
 \end{thm}

\begin{proof}
 For any integer $n$, Lemma \ref{l.ruelle_power} yields:
 $$
 \forall x \in \mathcal{A}_k^\N , \  R^nV_0(x)=\sum_{j=0}^{|s_k^n(x_0)|-1} \frac{1}{\delta_{s_k}(\sigma^j(s_k^n(x))}.
 $$
 
 Moreover, Proposition \ref{thm.length_kbo} yields 
 $$
 \forall n \geq k, \  \forall j < |s_k^n(x_0)|, \  \delta_{s_k}(\sigma^js_k^n(x))=\sum_{l\in \mathcal{A}_k}|s_k^n(l)||x_{[0...p-1]}|_l+\sum_{l=0}^{n-1}|s_k^l(0)|-j.
 $$
 
 Let us also remark that for any integer $n$:
 $$
 \sum_{l=0}^{n-1}|s_k^l(0)|=\sum_{l=0}^{n-1}\left(\gamma_0\lambda^l +r_0(l)\right),
 $$
 thus
 $$
  \sum_{l=0}^{n-1}|s_k^l(0)|=\gamma_0\frac{\lambda^n-1}{\lambda-1} +  \sum_{l=0}^{n-1}r_0(l).
 $$
 Let us denote $r'_0(n)=\displaystyle\frac{\gamma_0}{1-\lambda}+  \sum_{l=0}^{n-1}r_0(l)$ to have:
 $$
 \sum_{l=0}^{n-1}|s_k^l(0)|=\gamma_0\frac{\lambda^n}{\lambda-1} + r'_0(n).
 $$
 
 Finally, we have for any integer $n\geq k$, for any configuration $x \in \mathcal{A}_k^\N$ with $\delta_{s_k}(x)=p$, and for any $j < |s_k^n(x_0)|$:
 $$
  \delta_{s_k}(\sigma^js_k^n(x))=\sum_{l\in \mathcal{A}_k}\left(|x_{[0..p-1]}|_l\left(\gamma_l \lambda^n +r_l(n)\right)\right) +r'_0(n)+\frac{\lambda^n\gamma_0}{\lambda-1}-j
 $$
 which is also equal to:
 $$
 \lambda^n\left(\frac{\gamma_0}{\lambda-1}+\sum_{l\in \mathcal{A}_k}\gamma_l |x_{[0..p-1]}|_l +\frac{r'_0(n)+\sum_{l\in \mathcal{A}_k} r_l(n)|x_{[0..p-1]}|_l}{\lambda^n}\right)-j.
 $$
 
 So we can write for any integer $n$ greater or equal to $k$ and for any configuration $x \in \mathcal{A}_k^\N$:
 $$
R^nV(x)= \frac{1}{\lambda^n}\sum_{j=0}^{\gamma_{x_0}\lambda^n+r_{x_0}(n)-1}\frac{1}{\frac{\gamma_0}{\lambda-1}+\sum_{l\in \mathcal{A}_k}\gamma_l |x_{[0..p-1]}|_l +\frac{r'_0(n)+\sum_{l\in \mathcal{A}_k} r_l(n)|x_{[0..p-1]}|_l}{\lambda^n}-\frac{j}{\lambda^n}}.
 $$
 Let us estimate the term:
 $$
 Q(x,n):=\frac{1}{\lambda^n}\sum_{j=\gamma_{x_0}\lambda^n+1}^{\gamma_{x_0}\lambda^n+r_{x_0}(n)-1}\frac{1}{\frac{\gamma_0}{\lambda-1}+\sum_{l\in \mathcal{A}_k}\gamma_l |x_{[0..p-1]}|_l +\frac{r'_0(n)+\sum_{l\in \mathcal{A}_k} r_l(n)|x_{[0..p-1]}|_l}{\lambda^n}-\frac{j}{\lambda^n}}.
 $$
 Remark that
 $$
 \frac{\gamma_0}{\lambda-1}+\sum_{l\in \mathcal{A}_k}\gamma_l |x_{[0..p-1]}|_l +\frac{r'_0(n)+\sum_{l\in \mathcal{A}_k} r_l(n)|x_{[0..p-1]}|_l}{\lambda^n}-\frac{\gamma_{x_0}\lambda^n+r_{x_0}(n)-1}{\lambda^n}>0
 $$
and that for any integer $j$ in $\llbracket \lfloor \gamma_{x_0}\lambda^n+1\rfloor+1, \gamma_{x_0}\lambda^n+r_{x_0}(n)-1 \rrbracket$, the quantity:
$$
\frac{1}{\frac{\gamma_0}{\lambda-1}+\sum_{l\in \mathcal{A}_k}\gamma_l |x_{[0..p-1]}|_l +\frac{r'_0(n)+\sum_{l\in \mathcal{A}_k} r_l(n)|x_{[0..p-1]}|_l}{\lambda^n}-\frac{j}{\lambda^n}}
$$
is at most
$$
\frac{1}{\frac{\gamma_0}{\lambda-1}+\sum_{l\in \mathcal{A}_k}\gamma_l |x_{[0..p-1]}|_l +\frac{r'_0(n)+\sum_{l\in \mathcal{A}_k} r_l(n)|x_{[0..p-1]}|_l}{\lambda^n}-\frac{ \gamma_{x_0}\lambda^n+1}{\lambda^n}}.
$$
So there exists a constant $C_x>0$, not depending on $n$, such that
$$
 Q(x,n)\leq \frac{1}{\lambda^n} \sum_{j=\gamma_{x_0}\lambda^n+1}^{\gamma_{x_0}\lambda^n+r_{x_0}(n)-1} C_x,
$$
hence
$$
Q(x,n)\leq \frac{C_x r_{x_0}(n)}{\lambda^n}
$$
so, for any configuration $x \in \mathcal{A}_k^\N$, the term $Q(x,n)$ goes to 0 as $n$ goes to infinity.

So let us write, for any configuration $x \in \mathcal{A}_k^\N$ and any integer $n$:
$$
R^nV_0(x)=\frac{1}{\lambda^n}\sum_{j=0}^{\gamma_{x_0}\lambda^n}\frac{1}{\frac{\gamma_0}{\lambda-1}+\sum_{l\in \mathcal{A}_k}\gamma_l |x_{[0..p-1]}|_l +\frac{r'_0(n)+\sum_{l\in \mathcal{A}_k} r_l(n)|x_{[0..p-1]}|_l}{\lambda^n}-\frac{j}{\lambda^n}} +Q(x,n).
$$

Remark that $\frac{r'_0(n)+\sum_{l\in \mathcal{A}_k} r_l(n|x_{[0..p-1]}|_l)}{\lambda^n}$ goes to 0 as $n$ goes to infinity and that the function:
$$
f:y \mapsto \frac{1}{\frac{\gamma_0}{\lambda-1}+\sum_{l\in \mathcal{A}_k}\gamma_l |x_{[0..p-1]}|_l +y}
$$ 
is Lipschitz since $\frac{\gamma_0}{\lambda-1}+\sum_{l\in \mathcal{A}_k}\gamma_l |x_{[0..p-1]}|_l>0$. We use now the following lemma from \cite{bhl}.

\begin{lemma}
 Let $a,\lambda$ be some positive real numbers and $f$ a Lipschitz function defined on a neighborhood of $[0,a]$. Let $\phi:\N\rightarrow \R$ be a real sequence such that $|\phi(n)|\leq C\theta^n$ with $C>0$ and $0<\theta<\lambda$. We have
 $$
 \lim_{n\rightarrow +\infty}\frac{1}{\lambda^n}\sum_{k=0}^{a\lambda^n}f\left( \frac{k+\phi(n)}{\lambda(n)} \right)=\int_{0}^{a}f(x)dx.
 $$
\end{lemma}

Thus, for any configuration $x \in \mathcal{A}_k^\N$, the sequence $ (R^n V_0 (x))_{n\in \N}$ has the same limit as the Riemann sum $\frac{1}{\lambda^n}\displaystyle\sum_{j=0}^{\gamma_{x_0}\lambda^n}f\left(\frac{j}{\lambda^n}\right)$:
$$
\lim_{n\rightarrow + \infty} R^n V_0 (x) = \int_0^{\gamma_{x_0}}  \frac{dt}{\frac{\gamma_0}{\lambda-1}+\sum_{l\in \mathcal{A}_k}\gamma_l |x_{[0..p-1]}|_l -t}
$$
hence, 
$$
\lim_{n\rightarrow + \infty} R^n V_0 (x) = \log\left(\frac{\frac{\gamma_0}{\lambda-1}+\sum_{l\in \mathcal{A}_k}\gamma_l |x_{[0..p-1]}|_l}{\frac{\gamma_0}{\lambda-1}+\sum_{l\in \mathcal{A}_k}\gamma_l |x_{[0..p-1]}|_l - \gamma_{x_0}}\right).
$$
From Remark \ref{r.eigenvector} we deduce Theorem \ref{t.renorm_baby_case}.
\end{proof}

\subsection{Proof of Theorem \ref{t.renorm}}

Let us now define a whole family of potentials in the following way:

$$
\forall x \in \mathcal{A}_k^\N,\, V(x)=\frac{g(x)}{\delta_{s_k}(x)^\alpha}+\frac{h(x)}{\delta_{s_k}(x)^\alpha},
$$
$g$ being a positive continuous function and $h$ being $0$ on $\Sigma_{s_k}$ and $\alpha >0$.

 Theorem \ref{t.renorm} is an improvement upon Theorem \ref{t.renorm_baby_case} for it extends its results to this whole family of potentials.
 
 Since we know Theorem \ref{t.renorm_baby_case} the only necessary ingredient missing to prove this Theorem is the following technical lemma:
 \begin{lemma}\label{l.technical_abel}
  Let $(X,\sigma)$ be a uniquely ergodic subshift whose unique invariant probability measure is denoted $\mu$. Let $f$ be a continous integrable function on $(0,1)$, let $g:X\rightarrow\R$ be a continous function on $X$. Then we have, uniformly in $x \in X$:
  $$
  \lim_{n\rightarrow +\infty} \frac1n\sum_{j=0}^n f\left(\frac{k}{n}\right)g(\sigma^k(x))=\int_0^1f(x)dx\, \int_X gd\mu
  $$
 \end{lemma}
For the proof of this lemma, we refer the reader to Lemma 3.10 in \cite{bhl}.

The cases $\alpha <1$ and $\alpha>1$ can easily be seen from the computations in the proof of Theorem \ref{t.renorm_baby_case}. If $\alpha<1$ then the Riemann sum diverges towards infinity because the exponent on the denominator is to small and if $\alpha >1$ the the Riemann sum is crushed towards $0$.

The only interesting case is when $\alpha =1$, then we apply Lemma \ref{l.technical_abel} to $s_k^n(x)$ which is possible because we have uniform convergence and we use the computations of Theorem \ref{t.renorm_baby_case}.

\begin{rem}
 This theorem states that the family of potentials of the form $V(x)=\frac{g(x)}{\delta_{s_k}(x)^\alpha}+\frac{h(x)}{\delta_{s_k}(x)^\alpha}$ is stable under the renormalization operator. Notice that the first order Taylor expansion of $U$ is also of the form $\frac{g(x)}{\delta_{s_k}(x)}+\frac{h(x)}{\delta_{s_k}(x)}$. Hence amongst this family of potentials which is of some interest to us, this particular fixed point for $R$ is the only on to span an attracting ``line'' which stays in this family of potentials.
\end{rem}

\section{Freezing phase transition: proof of Theorem \ref{t.fpt}}\label{s.fpt}

Recall that for a given potential $V$, we define the pressure function for every positive real number $\beta$ by:
$$
P(\beta)=\sup\left\{h_{\mu} + \beta\int_X V d\mu\right\}.
$$
We are interested in points of non analyticity of the pressure function. Such points are called phase transitions. It is also known that the pressure function has an asymptote of the form $-a\beta+b$ with $a$ and $b$ non negative real numbers. If it reaches its asymptote we speak of freezing phase transition. It is obvious that, having $V$ supported on $\A_k^\N \backslash \Sigma_{s_k}$ and $(\Sigma_{s_k},\sigma)$ being uniquely ergodic of entropy zero, if the pressure function has a freezing phase transition, this function being decreasing and convex, then necessarily the asymptote it reaches is the horizontal axis.

Theorem 3 from \cite{bhl} gives a set of sufficient conditions to have a freezing phase transition in the case of subshifts. Namely, in order to have a freezing phase transition, it is sufficient for the subshift to satisfy:
\begin{itemize}
 \item being linear recurrent, which means that there exists a constant $C>0$ such that for any $x$ in the subshift and for any word $w$ of size $n$ appearing in $x$, two consecutive occurrences of $w$ in $x$ are separated by a word of length at most $Cn$.
 \item having all bispecial words (see Definition \ref{d.bispecial})  of length $c.\lambda^n +o(\lambda^n)$, where $\lambda >1$ and $c$ is chosen in a finite set;
 \item having only bispecial words not overlapping each other for more than a fixed proportion than the smaller one. We recall that two words $u$ and $v$ overlap (with overlap $u \cap v$ )if we can write:
 \bigskip
  
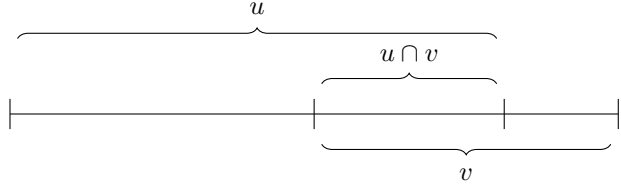
\begin{center}
     \begin{tikzpicture}
 \draw (0,0)--(8,0);
 \draw (0,0.2)--(0,-0.2);
 \draw (4,0.2)--(4,-0.2);
 \draw (8,0.2)--(8,-0.2);
 \draw (6.5,0.2)--(6.5,-0.2);
 \draw[decorate,decoration={brace, amplitude=4pt, raise=1cm}](0.1,0)--(6.4,0)node [black,midway,yshift=1.4cm]{$u$};
 \draw[decorate,decoration={brace, amplitude=4pt, raise=0.4cm}](4.1,0)--(6.4,0)node [black,midway,yshift=.8cm]{$u \cap v$}; 
 \draw[decorate,decoration={brace, amplitude=4pt, mirror, raise=0.4cm}](4.1,0)--(7.9,0)node [black,midway,yshift=-.8cm]{$v$};
\end{tikzpicture}
  \end{center}
  with $u \cap v$ of maximal size and different from $u$ and $v$.
\end{itemize}

The first condition is true since a subshift associated to a substitution is linearly recurrent \cite{durand}.

We remind the following property for $k$-bonacci substitutions.
\begin{prop}
For any $k\geq2$, the set of bispecial words words for the $k$-bonacci substitution is exactly:
 $$
\left\{s_k^n(0)s_k^{n-1}(0)...s_k(0)0,\, n \in \N\right\}.
$$
\end{prop}
\begin{proof}
It is an easy check that for any $n$, $s_k^n(0)s_k^{n-1}(0)...s_k(0)0$ is bispecial. Indeed, one can check that if $w$ is bispecial, then $s_k(w)0$ is bispecial and notice that $0$ is bispecial. So let us prove that if $w$ is bispecial, then there exists $n$ such that $w=s_k^n(0)s_k^{n-1}(0)...s_k(0)0$. First, notice that for $w$ to be left special, it must start with $0$. For $w$ to be right special, it has to end with $0$. Assume $w\neq 0$, then there is a unique word $v$  in $\lang_{s_k}$ such that $s_k(v)0=w$. Now one can check that $w$ being bispecial implies that $v$ is in turn bispecial, and the length of $v$ being lesser thant the length of $w$, iterating this procedure of ``desubstitution'' yields the result since the only bispecial word of length two or less is $0$.
\end{proof}

Remark \ref{r.perron} and the computations in the proof of Theorem \ref{t.renorm_baby_case} give the length of a bispecial word:
$$
 \sum_{l=0}^{n}|s_k^l(0)|=\gamma_0\frac{\lambda^{n+1}}{\lambda-1} + r'_0(n).
$$
This proves that the second property holds.

Finally,  to have a freezing phase transition, it is enough to know that there exists $c<1$ such that if $u$ and $v$ are overlapping bispecial words, then $|u \cap v | \leq c \min\{|u|,|v|\}.$ Here $u\cap v$ denotes the word of maximal size that is both a prefix of $v$ and a suffix of $u$.

 Let us define, for all $n$, the bispecial word $b_n=s_k^n(0)s_k^{n-1}(0)...s(0)0$. If two bispecials $u$ and $v$ overlap, then necessarily, the overlap $u \cap v$ is a bispecial word. We are not interested in the case where $u \cap v =u$ or $u \cap v =v$. Then let $u=b_n$ for a certain $n$, necessarily, $u \cap v=b_m$ for a certain $m$ smaller than $n$. So
 $$
 \frac{|u \cap v|}{|u|}=\frac{|b_m|}{|b_n|}
 $$
 and
 $$
 \frac{b_m}{b_n}\leq\frac{|b_m|}{|b_{m+1}|}.
 $$
 Now, applying Perron-Frobenius on the incidence matrix of $k$-bonacci substitution yields:
 $$
 \lim_{m\rightarrow +\infty} \frac{|b_m|}{|b_{m+1}|}= \frac{1}{\lambda}
 $$
 where $\lambda$ is the single dominating root of $X^k-\displaystyle\sum_{j=0}^{k-1}X^j$.
 
 The case  $ \frac{|u \cap v|}{|v|}$ is treated the same way and yields the same result.
 
 So the third property holds and this proves Theorem \ref{t.fpt}.
  
\appendix
\section*{Appendix}

\section{Example of the Tribonacci substitution}\label{s.tribo}

In this section we give the direct proof of Proposition \ref{thm.length_kbo} in the case of the Tribonacci substitution, which is the $k=3$ case.

\begin{prop}\label{p.tribodesub}
 Any word in $\lang_{s_3}$ which starts by a $0$ and ends by either $1$ or $2$ has a unique preimage by the Tribonacci substitution.
\end{prop}

This is obviously true. Suffices to read from left to right:
\begin{itemize}
 \item Every $0$ marks the beginning of an image of a letter.
 \item The letter which comes after a $0$ gives a unique way of desubstituting.
\end{itemize}

\begin{prop}\label{p.tribolang}
 The only three letters word starting by $00$ in $\lang_{s_3}$ is $001$.
\end{prop}

\begin{proof}

 $000 \notin \lang_{s_3}$ since it is either equal to $s_3(222)$ or is a prefix of $s_3(220)$ or $s_3(221)$. Either way, $000 \in \lang_{s_3}$ would imply that $22 \in \lang_{s_3}$ which is clearly not the case.
 $002 \notin \lang_{s_3}$ since $s_3(21)=002$ and $21$ is clearly not in the language.
 
\end{proof}

\begin{rem}\label{r.length_tribo}
 Let us remark right away that for any integer $n$:
 $$
 s_3^{n+3}(0)=s_3^{n+2}(0)s_3^{n+1}(0)s_3^{n}(0)
 $$
 which can also be written, for any positive integer $n$
 $$
 s_3^{n+3}(2)=s_3^{n+2}(2)s_3^{n+1}(2)s_3^{n}(2).
 $$
 
\end{rem}

\begin{prop}\label{p.delta_tribo}
 Let $x \in \{0,1\}^\N $ such that $\delta (x) = p$ then for any $n$ in $\N^*$, the maximal prefix of $s_3^n(x)$ in $\lang_{s_3}$ is:
 $$
 s_3^n(x_{[0...p-1]})s_3^{n-1}(0)...0.
 $$
 \end{prop}

\begin{proof}
 Let $x$ be a configuration on the alphabet $\{0,1,2\}$ such that the prefix $x_{[0...p-1]} \in \lang_{s_3}$ but $x_{[0...p]} \notin \lang_{s_3}$. We only prove the proposition for $n=1$, the general case being an immediate consequence.
 
 We have three cases to treat.
 \begin{itemize}
  \item If $x_p=0$, then $x_{[0...p-1]}0 \notin \lang_{s_3}$ hence $x_{[0...p-1]}1 \in \lang_{s_3}$ or $x_{[0...p-1]}2 \in \lang_{s_3}$ because a language defined by a substitution is extendable.
  Moreover 
  $$
  s_3(x)=s_3(x_{[0...p-1]})s_3(0)...
  $$
  so
  $$
  s_3(x)=s_3(x_{[0...p-1]})01...
  $$
  and $s_3(x_{[0...p-1]})0$ is a prefix of $s_3(x_{[0...p-1]}1)$ and $s_3(x_{[0...p-1]}2)$.
 As a consequence, $s_3(x_{[0...p-1]})0 \in \lang_{s_3}$. However $s_3(x_{[0...p-1]})01 \notin \lang_{s_3}$ because otherwise $x_{[0...p-1]}0$ would be in the language by Proposition \ref{p.tribodesub}. Finally the maximal prefix of $s_3(x)$ in the language is $s_3(x_{[0..p-1]})0$.
 \item We can treat the case $x_p=1$ in a similar way since a word starting by 0 and ending by a $2$ has a unique preimage by Tribonacci substitution.
  \item If $x_p=2$, then $x_{[0...p-1]}2 \notin \lang_{s_3}$ hence $x_{[0...p-1]}0 \in \lang_{s_3}$ or $x_{[0...p-1]}1 \in \lang_{s_3}$.
   Moreover 
  $$
  s_3(x)=s_3(x_{[0...p-1]})s_3(2)s_3(x_{p+1})...
  $$
  so
  $$
  s_3(x)=s_3(x_{[0...p-1]})00...
  $$
  and $s_3(x_{[0...p-1]})0$ is both a prefix of $s_3(x_{[0...p-1]}0)$ and $s_3(x_{[0...p-1]}1)$, one of which is in the language. As a consequence, $s_3(x_{[0...p-1]})0 \in \lang_{s_3}.$ However, $s_3(x_{[0...p-1]})00 \notin \lang_{s_3}$. Indeed let us suppose that  $s_3(x_{[0...p-1]})00$ is in $\lang_{s_3}$. Then, by Proposition \ref{p.tribolang}, necessarily $s_3(x_{[0...p-1]})001$ would be in $\lang_{s_3}$ which is absurd because we can uniquely desubstitute to find out that $s_3(x_{[0...p-1]}20) \in \lang_{s_3}$ which is a contradiction because then $x_{[0...p-1]}2$ would be in $\lang_{s_3}$.
 \end{itemize}
Finally, iterating the application yields that for any positive integer $n$, the maximal prefix  in the language $\lang_{s_3}$ of $s_3^n(x)$, where $x$  is defined as before,  is:
$$
s_3^n(x_{[0..p-1]})s_3^{n-1}(0)...0.
$$
\end{proof}

Let us now prove the following lemma: 

\begin{lemma}\label{prop.recog_tribo}
 $$
 \forall n \geq 3,\ \forall d \in \N, \  \omega^{s_3}_{[d...d+t_n-1]}=s_3^n(0) \Leftrightarrow d \in D_{s_3}^n.
 $$
 where $\omega^{s_3}$ is the fixed point in $\{0,1,2\}^\N$ for the Tribonacci substitution and $t_n$ is the length of $s_2^n(0)$.
\end{lemma}

\begin{proof}
 Let us first prove the following assertion:
  $$
 \forall n \geq 3, \forall d \in \N, \ d \in D_{s_3}^n  \Rightarrow \omega^{s_3}_{[d...d+t_n-1]}=s_3^n(0) .
 $$
 
 Let $n$ be an integer greater than 2 and $d$ be in $D_{s_3}^n$. Then $\omega^{s_3}_{[d...d+t_n-1]}$ is either equal to $s_3^n(0)$ or is a prefix of $s_3^n(1)s_3^n(0)$ or $s_3^n(2)s_3^n(0)$ since $ab \in \lang_{s_3}$ if and only if $a=0$ or $b=0$.
 \begin{itemize}
  \item If $\omega^{s_3}_{[d...d+t_n-1]}=s_3^n(0)$ then there is nothing to prove.
  \item If $\omega^{s_3}_{[d...d+t_n-1]}$ is a prefix of $s_3^n(1)s_3^n(0)$, then we write:
  $$
  s_3^n(1)s_3^n(0)=s_3^{n-1}(0)s_3^{n-1}(2)s_3^n(0),
  $$
  which can then be written
  $$
  s_3^n(1)s_3^n(0)=s_3^{n-1}(0)s_3^{n-2}(0)s_3^n(0),
  $$
  and noticing that $s_3^{n-3}(0)$ is a prefix of $s_3^n(0)$ yields
  $$
  \omega^{s_3}_{[d...d+t_n-1]}=s_3^n(0).
  $$
  \item If $\omega^{s_3}_{[d...d+t_n-1]}$ is a prefix of $s_3^n(2)s_3^n(0)$, then we write:
  $$
  s_3^n(2)s_3^n(0)=s_3^{n-1}(0)s_3^{n-1}(0)s_3^{n-1}(1)
  $$
  which can be written as
  $$
  s_3^n(2)s_3^n(0)=s_3^{n-1}(0)s_3^{n-2}(0)s_3^{n-2}(1)s_3^{n-1}(1),
  $$
  then again
  $$
  s_3^n(2)s_3^n(0)=s_3^{n-1}(0)s_3^{n-2}(0)s_3^{n-3}(0)s_3^{n-3}(2)s_3^{n-1}(1),
  $$
  so finally
  $$
  s_3^n(2)s_3^n(0)=s_3^n(0)s_3^{n-3}(2)s_3^{n-1}(1),
  $$
  hence 
   $$
  \omega^{s_3}_{[d...d+t_n-1]}=s_3^n(0).
  $$
 \end{itemize}

 The first step of the proof is thus complete.
 
 Let us now prove by induction on $n$ that 
  $$
 \forall n \geq 3,\   \forall d \in \N, \  \omega^{s_3}_{[d...d+t_n-1]}=s_3^n(0) \Rightarrow d \in D_{s_3}^n.
 $$
  
 First, we prove that this is true for the word $s_3^3(0)$.
 $$
 s_3^3(0)=0102010
 $$
 Let us also write the other images 
 $$
 s_3^3(1)=010201, \  s_3^3(2)=0102.
 $$
 Remark that $s_3^3(0)$ is of maximal length amongst the third power images of letters. Hence any occurence of this word contains at least one point in $D_{s_3}^3$, which we will represent by a point before the letter which has coordinate in $D_{s_3}^3$. Remark that all third power images contain one and only one letter $2$. Hence the only possibilities are:
 \begin{itemize}
  \item .0102010.
  \item .010201.0
  \item .0102.010
 \end{itemize}
 In any case, every occurence of $s_3^3(0)$ has starting position in $D_{s_3}^3$.
 This completes the initialisation.
 
 Let us now assume that this property is true for a given integer $n \geq 3$ and prove it for $n+1$.
 We represent the word $s_3^{n+1}(0)$ by a segment.
 
 \bigskip
  
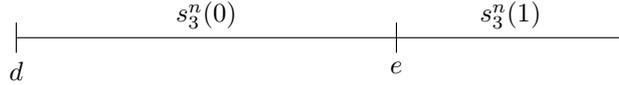
\begin{center}
     \begin{tikzpicture}
 \draw (0,0)--(8,0);
 \draw (0,0.2)--(0,-0.2);
 \draw (5,0.2)--(5,-0.2);
 \draw (8,0.2)--(8,-0.2);
 \draw (2.5,0) node [above] {$s_3^n(0)$};
 \draw (6.5,0) node [above] {$s_3^n(1)$};
 \draw (0,-0.2) node [below] {$d$};
 \draw (5,-0.2) node [below] {$e$};
\end{tikzpicture}
  \end{center}
 By induction hypothesis, $d$ is in $D_{s_3}^n$. We have three cases to treat.
 \begin{itemize}
  \item If $e$ is in $D_{s_3}^n$ then either $d$ or $e$ is in $D_{s_3}^{n+1}$ because they are two consecutive points of $D_{s_3}^n$ and the images of letters by the substitution $s_3$ are either of length 1 or 2.
  
  Moreover, if $e$ is in $D_{s_3}^{n+1}$, then writing $s_3^{n+1}(0)$ in the following way:
  \bigskip
  
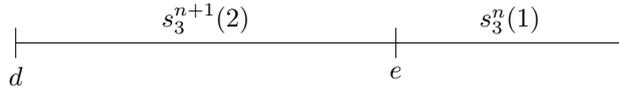
\begin{center}
     \begin{tikzpicture}
 \draw (0,0)--(8,0);
 \draw (0,0.2)--(0,-0.2);
 \draw (5,0.2)--(5,-0.2);
 \draw (8,0.2)--(8,-0.2);
 \draw (2.5,0) node [above] {$s_3^{n+1}(2)$};
 \draw (6.5,0) node [above] {$s_3^n(1)$};
 \draw (0,-0.2) node [below] {$d$};
 \draw (5,-0.2) node [below] {$e$};
\end{tikzpicture}
  \end{center}
  is enough to conclude that $d$ is in $D_{s_3}^{n+1}$.
  \item We can write $s_3^{n+1}(0)$ in the following way:
    
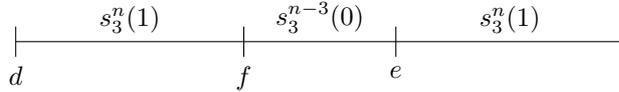
\begin{center}
     \begin{tikzpicture}
 \draw (0,0)--(8,0);
 \draw (0,0.2)--(0,-0.2);
 \draw (3,0.2)--(3,-0.2);
 \draw (5,0.2)--(5,-0.2);
 \draw (8,0.2)--(8,-0.2);
 \draw (1.5,0) node [above] {$s_3^n(1)$};
 \draw (4,0) node [above] {$s_3^{n-3}(0)$};
 \draw (6.5,0) node [above] {$s_3^n(1)$};
 \draw (0,-0.2) node [below] {$d$};
  \draw (3,-0.2) node [below] {$f$};
 \draw (5,-0.2) node [below] {$e$};
\end{tikzpicture}
  \end{center}
  and assume that $f$ is in $D_{s_3}^{n}$.
  But then we would have $s_3^{n-3}(0)s_3^n(1)=s_3^n(0)$ which is impossible because the last letter differs.
  \item Finally we can write $s_3^{n+1}(0)$ in this way:
  
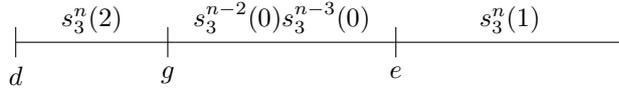
\begin{center}
  \begin{tikzpicture}
  \draw (0,0)--(8,0);
  \draw (0,0.2)--(0,-0.2);
  \draw (2,0.2)--(2,-0.2);
  \draw (5,0.2)--(5,-0.2);
  \draw (8,0.2)--(8,-0.2);
  \draw (1,0) node [above] {$s_3^n(2)$};
  \draw (3.5,0) node [above] {$s_3^{n-2}(0)s_3^{n-3}(0)$};
  \draw (6.5,0) node [above] {$s_3^n(1)$};
  \draw (0,-0.2) node [below] {$d$};
  \draw (2,-0.2) node [below] {$g$};
  \draw (5,-0.2) node [below] {$e$};
  \end{tikzpicture}
  \end{center}
  Or still:
   
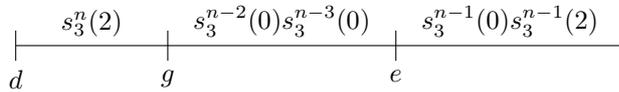
\begin{center}
  \begin{tikzpicture}
  \draw (0,0)--(8,0);
  \draw (0,0.2)--(0,-0.2);
  \draw (2,0.2)--(2,-0.2);
  \draw (5,0.2)--(5,-0.2);
  \draw (8,0.2)--(8,-0.2);
  \draw (1,0) node [above] {$s_3^n(2)$};
  \draw (3.5,0) node [above] {$s_3^{n-2}(0)s_3^{n-3}(0)$};
  \draw (6.5,0) node [above] {$s_3^{n-1}(0)s_3^{n-1}(2)$};
  \draw (0,-0.2) node [below] {$d$};
  \draw (2,-0.2) node [below] {$g$};
  \draw (5,-0.2) node [below] {$e$};
  \end{tikzpicture}
  \end{center}
  and assuming that $g$ is in $D_{s_3}^n$ would yield that
  $$
  s_3^n(0)=s_3^{n-2}(0)s_3^{n-3}(0)s_3^{n-1}(0)
  $$
  which is impossible.
 \end{itemize}
 Finally, only the first case is possible and we always have $d \in D_{s_3}^{n+1}$ which ends the proof.

 \end{proof}

\end{document}